\documentclass[12pt,a4paper]{amsart}
\usepackage{amssymb,amsxtra}
\usepackage[english,french]{babel}

\newtheorem{thm}{Theorem}[section]
\newtheorem{lem}[thm]{Lemma}

\theoremstyle{definition}

\theoremstyle{remark}

\theoremstyle{plain}

\theoremstyle{remark}

\newtheorem*{example}{Example}
\numberwithin{equation}{section}

\begin{document}

\title{ \textbf{Enumeration of some particular  $ 2n\times 10 $  n-times  Persymmetric  Matrices over $\mathbb{F}_{2} $ by rank}}
\author{Jorgen~Cherly}
\address{D\'epartement de Math\'ematiques, Universit\'e de
    Brest, 29238 Brest cedex~3, France}
\email{Jorgen.Cherly@univ-brest.fr}
\email{andersen69@wanadoo.fr}

\maketitle 
\begin{abstract}
Dans cet article nous comptons le nombre de certaines $2n\times 10$  n-fois matrices persym\' etriques de rang i sur $ \mathbb {F} _ {2} . $

 \end{abstract}

\selectlanguage{english}

\begin{abstract}  In this paper we count the number of some particular $2n\times 10$ n-times persymmetric rank i matrices over  $ \mathbb{F}_{2}.$
 \end{abstract}
 
   \maketitle 
\newpage
\tableofcontents
\newpage

   \section{Introduction.}
  \label{sec 1}  
  In this paper we propose to compute  the number $ \Gamma_{i}^{\left[2\atop{\vdots \atop 2}\right]\times 10}$  of rank i  $2n\times 10 $
   n-times persymmetric matrices over $\mathbb{F}_{2}$ of the below form for $0\leqslant i\leqslant \inf(2n,10) $  \\
    \begin{equation}
  \label{eq 3.1}
   \left (  \begin{array} {cccccccccc}
\alpha  _{1}^{(1)} & \alpha  _{2}^{(1)}  &   \alpha_{3}^{(1)} &   \alpha_{4}^{(1)} &   \alpha_{5}^{(1)} &  \alpha_{6}^{(1)}  & \alpha_{7}^{(1)}  & \alpha_{8}^{(1)} & \alpha_{9}^{(1)} & \alpha_{10}^{(1)}\\
\alpha  _{2}^{(1)} & \alpha  _{3}^{(1)}  &   \alpha_{4}^{(1)} &   \alpha_{5}^{(1)} &   \alpha_{6}^{(1)} &  \alpha_{7}^{(1)} & \alpha_{8}^{(1)}  &  \alpha_{9}^{(1)}& \alpha_{10}^{(1)} &\alpha_{11}^{(1)} \\ 
\hline \\
\alpha  _{1}^{(2)} & \alpha  _{2}^{(2)}  &   \alpha_{3}^{(2)} &   \alpha_{4}^{(2)} &   \alpha_{5}^{(2)} &  \alpha_{6}^{(2)} &  \alpha_{7}^{(2)}  & \alpha_{8}^{(2)} & \alpha_{9}^{(2)} &\alpha_{10}^{(2)} \\
\alpha  _{2}^{(2)} & \alpha  _{3}^{(2)}  &   \alpha_{4}^{(2)} &   \alpha_{5}^{(2)}&   \alpha_{6}^{(2)} &  \alpha_{7}^{(2)}  &  \alpha_{8}^{(2)} &  \alpha_{9}^{(2)} &\alpha_{10}^{(2)} &\alpha_{11}^{(2)} \\ 
\hline\\
\alpha  _{1}^{(3)} & \alpha  _{2}^{(3)}  &   \alpha_{3}^{(3)}  &   \alpha_{4}^{(3)} &   \alpha_{5}^{(3)} &  \alpha_{6}^{(3)} &  \alpha_{7}^{(3)} &  \alpha_{8}^{(3)} &\alpha_{9}^{(3)}& \alpha_{10}^{(3)} \\
\alpha  _{2}^{(3)} & \alpha  _{3}^{(3)}  &   \alpha_{4}^{(3)}&   \alpha_{5}^{(3)} &   \alpha_{6}^{(3)}  &  \alpha_{7}^{(3)} & \alpha_{8}^{(3)}  &   \alpha_{9}^{(3)}&\alpha_{10}^{(3)} &\alpha_{11}^{(3)} \\ 
\hline \\
\vdots & \vdots & \vdots  & \vdots  & \vdots & \vdots  & \vdots & \vdots  & \vdots & \vdots \\
\hline \\
\alpha  _{1}^{(n)} & \alpha  _{2}^{(n)}  &   \alpha_{3}^{(n)} &   \alpha_{4}^{(n)} &   \alpha_{5}^{(n)}  &  \alpha_{6}^{(n)} &  \alpha_{7}^{(n)}  &  \alpha_{8}^{(n)}& \alpha_{9}^{(n)}& \alpha_{10}^{(n)}  \\
\alpha  _{2}^{(n)} & \alpha  _{3}^{(n)}  &   \alpha_{4}^{(n)}&   \alpha_{5}^{(n)} &   \alpha_{6}^{(n)}  &  \alpha_{7}^{(n)} &  \alpha_{8}^{(n)}  &   \alpha_{9}^{(n)}&\alpha_{10}^{(n)} & \alpha_{11}^{(n)} \\ 
\end{array} \right )  
\end{equation} 
We remark that the results in this  paper are just a generalization of  the results  obtained in  the author's paper [14].
  \section{Some notations concerning the field of Laurent Series $ \mathbb{F}_{2}((T^{-1})) $. }
  \label{sec 2}  
  We denote by $ \mathbb{F}_{2}\big(\big({T^{-1}}\big) \big)
 = \mathbb{K} $ the completion
 of the field $\mathbb{F}_{2}(T), $  the field of  rational fonctions over the
 finite field\; $\mathbb{F}_{2}$,\; for the  infinity  valuation \;
 $ \mathfrak{v}=\mathfrak{v}_{\infty }$ \;defined by \;
 $ \mathfrak{v}\big(\frac{A}{B}\big) = degB -degA $ \;
 for each pair (A,B) of non-zero polynomials.
 Then every element non-zero t in
  $\mathbb{F}_{2}\big(\big({\frac{1}{T}}\big) \big) $
 can be expanded in a unique way in a convergent Laurent series
                              $  t = \sum_{j= -\infty }^{-\mathfrak{v}(t)}t_{j}T^j
                                 \; where\; t_{j}\in \mathbb{F}_{2}. $\\
  We associate to the infinity valuation\; $\mathfrak{v}= \mathfrak{v}_{\infty }$
   the absolute value \; $\vert \cdot \vert_{\infty} $\; defined by \;
  \begin{equation*}
  \vert t \vert_{\infty} =  \vert t \vert = 2^{-\mathfrak{v}(t)}. \\
\end{equation*}
    We denote  E the  Character of the additive locally compact group
$  \mathbb{F}_{2}\big(\big({\frac{1}{T}}\big) \big) $ defined by \\
\begin{equation*}
 E\big( \sum_{j= -\infty }^{-\mathfrak{v}(t)}t_{j}T^j\big)= \begin{cases}
 1 & \text{if      }   t_{-1}= 0, \\
  -1 & \text{if      }   t_{-1}= 1.
    \end{cases}
\end{equation*}
  We denote $\mathbb{P}$ the valuation ideal in $ \mathbb{K},$ also denoted the unit interval of  $\mathbb{K},$ i.e.
  the open ball of radius 1 about 0 or, alternatively, the set of all Laurent series 
   $$ \sum_{i\geq 1}\alpha _{i}T^{-i}\quad (\alpha _{i}\in  \mathbb{F}_{2} ) $$ and, for every rational
    integer j,  we denote by $\mathbb{P}_{j} $
     the  ideal $\left\{t \in \mathbb{K}|\; \mathfrak{v}(t) > j \right\}. $
     The sets\; $ \mathbb{P}_{j}$\; are compact subgroups  of the additive
     locally compact group \; $ \mathbb{K}. $\\
      All $ t \in \mathbb{F}_{2}\Big(\Big(\frac{1}{T}\Big)\Big) $ may be written in a unique way as
$ t = [t] + \left\{t\right\}, $ \;  $  [t] \in \mathbb{F}_{2}[T] ,
 \; \left\{t\right\}\in \mathbb{P}  ( =\mathbb{P}_{0}). $\\
 We denote by dt the Haar measure on \; $ \mathbb{K} $\; chosen so that \\
  $$ \int_{\mathbb{P}}dt = 1. $$\\
  
  $$ Let \quad
  (t_{1},t_{2},\ldots,t_{n} )
 =  \big( \sum_{j=-\infty}^{-\nu(t_{1})}\alpha _{j}^{(1)}T^{j},  \sum_{j=-\infty}^{-\nu(t_{2})}\alpha _{j}^{(2)}T^{j} ,\ldots, \sum_{j=-\infty}^{-\nu(t_{n})}\alpha _{j}^{(n)}T^{j}\big) \in  \mathbb{K}^{n}. $$ 
 We denote $\psi  $  the  Character on  $(\mathbb{K}^n, +) $ defined by \\
 \begin{align*}
  \psi \big( \sum_{j=-\infty}^{-\nu(t_{1})}\alpha _{j}^{(1)}T^{j},  \sum_{j=-\infty}^{-\nu(t_{2})}\alpha _{j}^{(2)}T^{j} ,\ldots, \sum_{j=-\infty}^{-\nu(t_{n})}\alpha _{j}^{(n)}T^{j}\big) & = E \big( \sum_{j=-\infty}^{-\nu(t_{1})}\alpha _{j}^{(1)}T^{j}\big) \cdot E\big( \sum_{j=-\infty}^{-\nu(t_{2})}\alpha _{j}^{(2)}T^{j}\big)\cdots E\big(  \sum_{j=-\infty}^{-\nu(t_{n})}\alpha _{j}^{(n)}T^{j}\big) \\
  & = 
    \begin{cases}
 1 & \text{if      }     \alpha _{-1}^{(1)} +    \alpha _{-1}^{(2)}  + \ldots +   \alpha _{-1}^{(n)}   = 0 \\
  -1 & \text{if      }    \alpha _{-1}^{(1)} +    \alpha _{-1}^{(2)}  + \ldots +   \alpha _{-1}^{(n)}   =1                                                                                                                          
    \end{cases}
  \end{align*}
  \section{Some results concerning  n-times persymmetric matrices over  $ \mathbb{F}_{2}$.}
  \label{sec 3}  
     $$ Set\quad
  (t_{1},t_{2},\ldots,t_{n} )
 =  \big( \sum_{i\geq 1}\alpha _{i}^{(1)}T^{-i}, \sum_{i \geq 1}\alpha  _{i}^{(2)}T^{-i},\sum_{i \geq 1}\alpha _{i}^{(3)}T^{-i},\ldots,\sum_{i \geq 1}\alpha _{i}^{(n)}T^{-i}   \big) \in  \mathbb{P}^{n}. $$

     Denote by $D^{\left[2 \atop{\vdots \atop 2}\right]\times k}(t_{1},t_{2},\ldots,t_{n} ) $
    
    the following $2n \times k $ \;  n-times  persymmetric  matrix  over the finite field  $\mathbb{F}_{2}. $ 
    
  \begin{equation}
  \label{eq 3.1}
   \left (  \begin{array} {cccccccc}
\alpha  _{1}^{(1)} & \alpha  _{2}^{(1)}  &   \alpha_{3}^{(1)} &   \alpha_{4}^{(1)} &   \alpha_{5}^{(1)} &  \alpha_{6}^{(1)}  & \ldots  &  \alpha_{k}^{(1)} \\
\alpha  _{2}^{(1)} & \alpha  _{3}^{(1)}  &   \alpha_{4}^{(1)} &   \alpha_{5}^{(1)} &   \alpha_{6}^{(1)} &  \alpha_{7}^{(1)} & \ldots  &  \alpha_{k+1}^{(1)} \\ 
\hline \\
\alpha  _{1}^{(2)} & \alpha  _{2}^{(2)}  &   \alpha_{3}^{(2)} &   \alpha_{4}^{(2)} &   \alpha_{5}^{(2)} &  \alpha_{6}^{(2)} & \ldots   &  \alpha_{k}^{(2)} \\
\alpha  _{2}^{(2)} & \alpha  _{3}^{(2)}  &   \alpha_{4}^{(2)} &   \alpha_{5}^{(2)}&   \alpha_{6}^{(2)} &  \alpha_{7}^{(2)}  & \ldots  &  \alpha_{k+1}^{(2)} \\ 
\hline\\
\alpha  _{1}^{(3)} & \alpha  _{2}^{(3)}  &   \alpha_{3}^{(3)}  &   \alpha_{4}^{(3)} &   \alpha_{5}^{(3)} &  \alpha_{6}^{(3)} & \ldots  &  \alpha_{k}^{(3)} \\
\alpha  _{2}^{(3)} & \alpha  _{3}^{(3)}  &   \alpha_{4}^{(3)}&   \alpha_{5}^{(3)} &   \alpha_{6}^{(3)}  &  \alpha_{7}^{(3)} & \ldots  &  \alpha_{k+1}^{(3)} \\ 
\hline \\
\vdots & \vdots & \vdots  & \vdots  & \vdots & \vdots  & \vdots & \vdots \\
\hline \\
\alpha  _{1}^{(n)} & \alpha  _{2}^{(n)}  &   \alpha_{3}^{(n)} &   \alpha_{4}^{(n)} &   \alpha_{5}^{(n)}  &  \alpha_{6}^{(n)} & \ldots  &  \alpha_{k}^{(n)} \\
\alpha  _{2}^{(n)} & \alpha  _{3}^{(n)}  &   \alpha_{4}^{(n)}&   \alpha_{5}^{(n)} &   \alpha_{6}^{(n)}  &  \alpha_{7}^{(n)} & \ldots  &  \alpha_{k+1}^{(n)} \\ 
\end{array} \right )  
\end{equation} 
We denote by  $ \Gamma_{i}^{\left[2\atop{\vdots \atop 2}\right]\times k}$  the number of rank i n-times persymmetric matrices over $\mathbb{F}_{2}$ of the above form :  \\

  Let $ \displaystyle  f (t_{1},t_{2},\ldots,t_{n} ) $  be the exponential sum  in $ \mathbb{P}^{n} $ defined by\\
    $(t_{1},t_{2},\ldots,t_{n} ) \displaystyle\in \mathbb{P}^{n}\longrightarrow \\
    \sum_{deg Y\leq k-1}\sum_{deg U_{1}\leq  1}E(t_{1} YU_{1})
  \sum_{deg U_{2} \leq 1}E(t _{2} YU_{2}) \ldots \sum_{deg U_{n} \leq 1} E(t _{n} YU_{n}). $\vspace{0.5 cm}\\
    Then
  $$     f_{k} (t_{1},t_{2},\ldots,t_{n} ) =
  2^{2n+k- rank\big[ D^{\left[2\atop{\vdots \atop 2}\right]\times k}(t_{1},t_{2},\ldots,t_{n} )\big] } $$

    Hence  the number denoted by $ R_{q,n}^{(k)} $ of solutions \\
  
 $(Y_1,U_{1}^{(1)},U_{2}^{(1)}, \ldots,U_{n}^{(1)}, Y_2,U_{1}^{(2)},U_{2}^{(2)}, 
\ldots,U_{n}^{(2)},\ldots  Y_q,U_{1}^{(q)},U_{2}^{(q)}, \ldots,U_{n}^{(q)}   ) \in (\mathbb{F}_{2}[T])^{(n+1)q}$ \vspace{0.5 cm}\\
 of the polynomial equations  \vspace{0.5 cm}
  \[\left\{\begin{array}{c}
 Y_{1}U_{1}^{(1)} + Y_{2}U_{1}^{(2)} + \ldots  + Y_{q}U_{1}^{(q)} = 0  \\
    Y_{1}U_{2}^{(1)} + Y_{2}U_{2}^{(2)} + \ldots  + Y_{q}U_{2}^{(q)} = 0\\
    \vdots \\
   Y_{1}U_{n}^{(1)} + Y_{2}U_{n}^{(2)} + \ldots  + Y_{q}U_{n}^{(q)} = 0 
 \end{array}\right.\]
 
    $ \Leftrightarrow
    \begin{pmatrix}
   U_{1}^{(1)} & U_{1}^{(2)} & \ldots  & U_{1}^{(q)} \\ 
      U_{2}^{(1)} & U_{2}^{(2)}  & \ldots  & U_{2}^{(q)}  \\
\vdots &   \vdots & \vdots &   \vdots   \\
U_{n}^{(1)} & U_{n}^{(2)}   & \ldots  & U_{n}^{(q)} \\
 \end{pmatrix}  \begin{pmatrix}
   Y_{1} \\
   Y_{2}\\
   \vdots \\
   Y_{q} \\
  \end{pmatrix} =   \begin{pmatrix}
  0 \\
  0 \\
  \vdots \\
  0 
  \end{pmatrix} $\\

    satisfying the degree conditions \\
                   $$  degY_i \leq k-1 ,
                   \quad degU_{j}^{(i)} \leq 1, \quad  for \quad 1\leq j\leq n,  \quad 1\leq i \leq q $$ \\
  is equal to the following integral over the unit interval in $ \mathbb{K}^{n} $
    $$ \int_{\mathbb{P}^{n}} f_{k}^{q}(t_{1},t_{2},\ldots,t_{n}) dt_{1}dt _{2}\ldots dt _{n}. $$
  Observing that $ f (t_{1},t_{2},\ldots,t_{n} ) $ is constant on cosets of $ \prod_{j=1}^{n}\mathbb{P}_{k+1} $ in $ \mathbb{P}^{n} $\;
  the above integral is equal to 
  
  \begin{equation}
  \label{eq 3.2}
 2^{q(2n+k) - (k+1)n}\sum_{i = 0}^{\inf(2n,k)}
  \Gamma_{i}^{\left[2\atop{\vdots \atop 2}\right]\times k} 2^{-iq} =  R_{q,n}^{(k)}. 
 \end{equation}
 
 \begin{eqnarray}
 \label{eq 3.3}
\text{ Recall that $ R_{q,n}^{(k)}$ is equal to the number of solutions of the polynomial system} \nonumber \\
    \begin{pmatrix}
   U_{1}^{(1)} & U_{1}^{(2)} & \ldots  & U_{1}^{(q)} \\ 
      U_{2}^{(1)} & U_{2}^{(2)}  & \ldots  & U_{2}^{(q)}  \\
\vdots &   \vdots & \vdots &   \vdots   \\
U_{n}^{(1)} & U_{n}^{(2)}   & \ldots  & U_{n}^{(q)} \\
 \end{pmatrix}  \begin{pmatrix}
   Y_{1} \\
   Y_{2}\\
   \vdots \\
   Y_{q} \\
  \end{pmatrix} =   \begin{pmatrix}
  0 \\
  0 \\
  \vdots \\
  0 
  \end{pmatrix} \\
 \text{ satisfying the degree conditions}\nonumber \\
                     degY_i \leq k-1 ,
                   \quad degU_{j}^{(i)} \leq 1, \quad  for \quad 1\leq j\leq n,   \quad 1\leq i \leq q. \nonumber
 \end{eqnarray}

 From \eqref{eq 3.2} we obtain for q = 1\\
   \begin{align}
  \label{eq 3.4}
 2^{k-(k-1)n}\sum_{i = 0}^{\inf(2n,k)}
 \Gamma_{i}^{\left[2\atop{\vdots \atop 2}\right]\times k} 2^{-i} =  R_{1,n}^{(k)} = 2^{2n}+2^k-1.
  \end{align}

We have obviously \\

   \begin{align}
  \label{eq 3.5}
 \sum_{i = 0}^{k}
 \Gamma_{i}^{\left[2\atop{\vdots \atop 2}\right]\times k}  = 2^{(k+1)n}.  
 \end{align}

From  the fact that the number of rank one persymmetric  matrices over $\mathbb{F}_{2}$ is equal to three  we obtain using
 combinatorial methods  : \\
 
    \begin{equation}
  \label{eq 3.6}
 \Gamma_{1}^{\left[2\atop{\vdots \atop 2}\right]\times k}  = (2^{n}-1)\cdot 3.
  \end{equation}
  For more details see Cherly  [11]

 \subsection{Computation of $ \Gamma_{7}^{\left[2\atop{\vdots \atop 2}\right]\times k} $.}
We recall (see section \ref{sec 3} ) that $ \Gamma_{7}^{\left[2\atop{\vdots \atop 2}\right]\times k}$ denotes the number of rank 7
n-times persymmetric matrices over  $\mathbf{F}_2 $  of the form \eqref{eq 3.1}
We shall need the following Lemma  : \\
  \begin{lem}
\label{lem 3.1}
\begin{equation}
\label{eq 3.7}
   \Gamma_{7}^{\left[2\atop{\vdots \atop 2}\right]\times k}=   \begin{cases}
 0 & \text{if  } n = 0,  \\  
0 & \text{if  } n = 1,   \\
0 & \text{if  } n = 2,\\
0 & \text{if  } n = 3,\\
3720\cdot 2^{3k}-416640\cdot2^{2k}+13332480\cdot2^{k}-121896960  & \text{if   } n=4,\\
115320\cdot[2^{3k}+1148\cdot2^{2k}-2^7\cdot917\cdot2^{k}+311\cdot2^{13}]  & \text{if   }  n = 5, \\
 \end{cases}
   \end{equation} 
   
   \begin{equation}
\label{eq 3.8}
   \Gamma_{7}^{\left[2\atop{\vdots \atop 2}\right]\times k}=   \begin{cases}
 255\cdot 2^{7n}-381\cdot2^{6n}-31122\cdot2^{5n} 
    +105648\cdot 2^{4n}  \\+758880\cdot 2^{3n}-4617984 \cdot 2^{2n}+7913472 \cdot2^{n} -4128768    & \text{if   } k=8,\\
  255\cdot 2^{7n}  +42291\cdot2^{6n}  -219618\cdot 2^{5n}-4053808\cdot 2^{4n} \\ +
 32840160\cdot 2^{3n}-82168576 \cdot 2^{2n}+ 81543168\cdot2^{n} -27983872  & \text{if   } k=9,
   \end{cases}
   \end{equation} 
    \end{lem}
\begin{proof}
Lemma \ref{lem 3.1} follows from  Cherly[12,13 and 14].
\end{proof}

  \begin{lem}
  \label{lem 3.2}
We postulate that :\\
\begin{align}
\label{eq 3.9}
 \displaystyle   \Gamma_{7}^{\left[2\atop{\vdots \atop 2}\right]\times k} =   255\cdot 2^{7n}  +a(k) \cdot 2^{6n}  +b(k) \cdot 2^{5n}+ c(k) \cdot 2^{4n}\\
  + d(k)\cdot 2^{3n}+e(k) \cdot 2^{2n}  +f(k)\cdot 2^{n} +g(k)\nonumber \\ 
  =  255\cdot 2^{7n}  + [\frac{2667}{16}\cdot 2^k -43053]  \cdot 2^{6n} \nonumber\\
   +[ \frac{465}{32}\cdot2^{2k}-\frac{190341}{16}\cdot2^k+2062014] \cdot 2^{5n} \nonumber \\
  +  [ \frac{31}{168}\cdot 2^{3k}-\frac{45229}{96}\cdot2^{2k}+\frac{6262403}{24}\cdot2^k-\frac{817168432}{21}]\cdot 2^{4n}\nonumber  \\
  +[ -\frac{465}{168}\cdot2^{3k}+\frac{231105}{48}\cdot2^{2k}-\frac{4605205}{2}\cdot2^k+\frac{2247886880}{7}] \cdot 2^{3n} \nonumber\\
  + [\frac{155}{12}\cdot2^{3k}-\frac{233585}{12}\cdot2^{2k}+\frac{26162884}{3}\cdot2^k-\frac{3534612736}{3} ] \cdot 2^{2n}\nonumber \\
   + [-\frac{155}{7}\cdot2^{3k}+31310\cdot2^{2k}-13600384\cdot2^k+\frac{1466315\cdot2^{13}}{7} +11373\cdot2^{13}]\cdot 2^{n}\nonumber \\
   + \frac{248}{21}\cdot2^{3k}-\frac{48608}{3}\cdot2^{2k}+\frac{20798464}{3}\cdot2^k-\frac{293263\cdot2^{16}}{21}\nonumber \\
=  \frac{31}{168}\cdot [2^{4n}-15\cdot 2^{3n}+70 \cdot 2^{2n}-120\cdot 2^{n}+64] \cdot2^{3k}\nonumber\\
+  \frac{1}{96}\cdot [ 1395\cdot2^{5n}-45229\cdot2^{4n}+462210\cdot 2^{3n}-1868680 \cdot 2^{2n}+3005760\cdot 2^{n}-1555456] \cdot2^{2k}\nonumber\\
+  \frac{1}{48}\cdot [ 8001\cdot2^{6n}-571023\cdot2^{5n}+12524806\cdot2^{4n}-110524920\cdot 2^{3n}+418606144\cdot 2^{2n}\nonumber\\
-652818432\cdot 2^{n}+332775424] \cdot2^{k}\nonumber\\
+\frac{1}{21}\cdot [5355\cdot2^{7n}- 904113\cdot2^{6n}+43302294\cdot2^{5n}-817168432\cdot2^{4n}+6743660640\cdot 2^{3n}\nonumber\\
-96649567\cdot2^{8}\cdot 2^{2n}+4637778\cdot2^{13}\cdot 2^{n}-293263\cdot2^{16}] \nonumber
  \end{align}
\end{lem}
 \begin{proof}
 \begin{align*}
   \text{ From the expression of $ \Gamma_{7}^{\left[2\atop{\vdots \atop 2}\right]\times k} $  in  \eqref{eq 3.8} for k=8,9  we assume that }\\
 \Gamma_{7}^{\left[2\atop{\vdots \atop 2}\right]\times k} \quad \text{can be written in the form} :\\
   \displaystyle  255\cdot 2^{7n}  +a(k) \cdot 2^{6n}  +b(k) \cdot 2^{5n}+ c(k) \cdot 2^{4n}\\
  + d(k)\cdot 2^{3n}+e(k) \cdot 2^{2n}  +f(k)\cdot 2^{n} +g(k) \\ 
  \text{Set} \;  Y=2^n, \; \text{then}\; \Gamma_{7}^{\left[2\atop{\vdots \atop 2}\right]\times k}\\ =  255 \cdot Y^7 +a(k) \cdot Y^6  +b(k) \cdot Y^5  +c(k) \cdot Y^4+ d(k) \cdot Y^3 + e(k) \cdot Y^2+f(k) \cdot Y +g(k) \\
  \Gamma_{7}^{\left[2\atop{\vdots \atop 2}\right]\times k} = 0 \quad  \text{ for}\; n\in\{0,1,2,3\}\\
\text{Then} \; \Gamma_{7}^{\left[2\atop{\vdots \atop 2}\right]\times k} = (Y-1)(Y-2)(Y-4)(Y-8)[255 \cdot Y^3+\alpha(k)\cdot Y^2+\beta(k) \cdot Y+\gamma(k)] \\
 = [2^{4n}-15\cdot 2^{3n}+70\cdot 2^{2n}-120\cdot 2^n+64] \cdot [ 255 \cdot 2^{3n}+\alpha(k) \cdot 2^{2n}+\beta(k) \cdot 2^{n}+\gamma(k)] \\
  = 255\cdot 2^{7n}  + ( \alpha(k)-3825) \cdot 2^{6n}  +( \beta(k)-15\cdot \alpha(k)+17850) \cdot 2^{5n}\\
 +( \gamma(k)-15\cdot \beta(k)+70\cdot \alpha(k)-30600 ) \cdot 2^{4n} \\
 + ( -15\cdot \gamma(k)+70\cdot \beta(k)-120 \cdot \alpha(k)+16320)\cdot 2^{3n}\\
 +  (70\cdot \gamma(k)-120\cdot \beta(k)  +64\cdot \alpha(k) )\cdot 2^{2n} \\
  + (-120 \cdot \gamma(k) +64\cdot \beta(k)) \cdot 2^n + 64\cdot  \gamma(k)   \\ 
\end{align*}
We then deduce : \\

 \begin{equation}
 \label{eq 3.10}
   \begin{cases} 
 \displaystyle a(k) = \alpha(k)-3825 \\
  \displaystyle  b(k) = \beta(k)-15\cdot \alpha(k)+17850 \\
   \displaystyle c(k) = \gamma(k)-15\cdot \beta(k)+70\cdot \alpha(k)-30600  \\
    \displaystyle d(k) = -15\cdot \gamma(k)+70\cdot \beta(k)-120 \cdot \alpha(k)+16320\\
 \displaystyle e(k) = 70\cdot \gamma(k)-120\cdot \beta(k) +64\cdot \alpha(k)  \\
    \displaystyle f(k) = -120 \cdot \gamma(k) +64\cdot \beta(k) \\ 
  \displaystyle g(k) = 64\cdot  \gamma(k)  
\end{cases}
    \end{equation}
To compute the expression of  $ \Gamma_{7}^{\left[2\atop{\vdots \atop 2}\right]\times k} $ we need only to compute $\alpha(k),\beta(k)$ and $\gamma(k)$.\\
\textbf{Computation of $\alpha(k)$}\\
From the expressions of $\Gamma_{i}^{\left[2\atop{\vdots \atop 2}\right]\times k} $ for $2 \leqslant i \leqslant 6$ in $ \mathbf{Lemma \;3.3} \; [12]$\\
  we assume that a(k) can be written in the form $a\cdot2^{k}+b$.\\
  We obtain from \eqref{eq 3.8} 
 \begin{equation*}
   \begin{cases} 
    a(k) = a\cdot 2^k+b\\
     a(8) = a\cdot 256+b= -381\\
      a(9) = a\cdot 512+b= 42291\\
      a=\frac{2667}{16}\\
      b= -43053\\
     (1) \quad  \alpha(k) = a\cdot 2^k+b+3825= \frac{2667}{16}\cdot 2^k -39228
   \end{cases}
    \end{equation*}

\textbf{The case n=4.}  \\ 

$\Gamma_{7}^{\left[2\atop{\vdots \atop 2}\right]\times k} = (2^4-1)(2^4-2)(2^4-4)(2^4-8)[255 \cdot 2^{12}+\alpha(k)\cdot 2^8+\beta(k) \cdot 2^4+\gamma(k)] $\\
$= 20160\cdot [255 \cdot 2^{12}+\alpha(k)\cdot 2^8+\beta(k) \cdot 2^4+\gamma(k)] =3720\cdot 2^{3k}-416640\cdot2^{2k}+13332480\cdot2^{k}-121896960 $\\
$\Rightarrow 255 \cdot 2^{12}+\alpha(k)\cdot 2^8+\beta(k) \cdot 2^4+\gamma(k) = \frac{1}{20160}\cdot[ 3720\cdot 2^{3k}-416640\cdot 2^{2k}+13332480 \cdot2^{k}-121896960  ]$\\
$    \Rightarrow     (2) \quad       256\cdot \alpha(k)+16\cdot \beta(k)+\gamma (k)=\frac{1}{20160}\cdot[ 3720\cdot 2^{3k}-416640\cdot 2^{2k}+13332480 \cdot2^{k}-121896960  ]  -255 \cdot 2^{12} $\\
$    \Rightarrow     (2) \quad       256\cdot \alpha(k)+16\cdot \beta(k)+\gamma (k)=\frac{31}{168}\cdot[  2^{3k}-112\cdot 2^{2k}+3584 \cdot2^{k}-32768  ]  -255 \cdot 2^{12} $\\

\textbf{The case n=5.} \\

$\Gamma_{7}^{\left[2\atop{\vdots \atop 2}\right]\times k} = (2^5-1)(2^5-2)(2^5-4)(2^5-8)[255 \cdot 2^{15}+\alpha(k)\cdot 2^{10}+\beta(k) \cdot 2^5+\gamma(k)] $\\
$= 624960 \cdot [255 \cdot 2^{15}+\alpha(k)\cdot 2^{10}+\beta(k) \cdot 2^5+\gamma(k)]   = 115320\cdot[2^{3k}+1148\cdot2^{2k}-2^7\cdot917\cdot2^{k}+311\cdot2^{13}] $ \\
$\Rightarrow 255 \cdot 2^{15}+\alpha(k)\cdot 2^{10}+\beta(k) \cdot 2^5+\gamma(k) = \frac{1}{624960}\cdot[ 115320\cdot(2^{3k}+1148\cdot2^{2k}-2^7\cdot917\cdot2^{k}+311\cdot2^{13}) ] $\\
$    \Rightarrow     (3) \quad       1024\cdot \alpha(k)+32\cdot \beta(k)+\gamma (k)=\frac{1}{624960}\cdot[ 115320\cdot(2^{3k}+1148\cdot2^{2k}-2^7\cdot917\cdot2^{k}+311\cdot2^{13})]    -255 \cdot 2^{15} $\\
$    \Rightarrow     (3) \quad       1024\cdot \alpha(k)+32\cdot \beta(k)+\gamma (k)=\frac{31}{168}\cdot[ 2^{3k}+1148\cdot2^{2k}-117376\cdot2^{k}+2547712]    -255 \cdot 2^{15} $\\
We then obtain :\\
 \begin{equation}
 \label{eq 3.11}
   \begin{cases} 
  (1) \quad  \alpha(k) =  \frac{2667}{16}\cdot 2^k -39228\\
(2) \quad    256\cdot \alpha(k)+16\cdot \beta(k)+\gamma (k)\\
=\frac{31}{168}\cdot[  2^{3k}-112\cdot 2^{2k}+3584 \cdot2^{k}-32768  ]  -255 \cdot 2^{12} \\
 (3)\quad     1024\cdot \alpha(k)+32\cdot \beta(k)+\gamma (k)\\
 =\frac{31}{168}\cdot[ 2^{3k}+1148\cdot2^{2k}-117376\cdot2^{k}+2547712]    -255 \cdot 2^{15} 
  \end{cases}
    \end{equation} 

From \eqref{eq 3.11} we deduce : \\
 \begin{equation}
 \label{eq 3.12}
   \begin{cases} 
     \alpha(k) =  \frac{2667}{16}\cdot 2^k -39228\\
     \beta(k) =  \frac{465}{32}\cdot2^{2k} -9396\cdot 2^k +1455744\\
       \gamma (k)= \frac{31}{168}\cdot 2^{3k}-\frac{1519}{6}\cdot2^{2k}+\frac{324976}{3}\cdot 2^k-\frac{300301312}{21}
   \end{cases}
    \end{equation} 

Combining \eqref{eq 3.10} and \eqref{eq 3.12} we get :Ê\\

 \begin{equation}
 \label{eq 3.13}
   \begin{cases} 
 \displaystyle a(k) = \alpha(k)-3825=   \frac{2667}{16}\cdot 2^k -43053 \\
  \displaystyle  b(k) = \beta(k)-15\cdot \alpha(k)+17850= \\
   \frac{465}{32}\cdot2^{2k} -9396\cdot 2^k +1455744-15\cdot [\frac{2667}{16}\cdot 2^k -39228]+17850\\
   =  \frac{465}{32}\cdot2^{2k}-\frac{190341}{16}\cdot2^k+2062014\\
   \displaystyle c(k) = \gamma(k)-15\cdot \beta(k)+70\cdot \alpha(k)-30600 \\
   = \frac{31}{168}\cdot 2^{3k}-\frac{1519}{6}\cdot2^{2k}+\frac{324976}{3}\cdot 2^k-\frac{300301312}{21}\\
   -15\cdot [\frac{465}{32}\cdot2^{2k} -9396\cdot 2^k +1455744]+70\cdot [\frac{2667}{16}\cdot 2^k -39228]-30600\\
   =  \frac{31}{168}\cdot 2^{3k}-\frac{45229}{96}\cdot2^{2k}+\frac{6262403}{24}\cdot2^k-\frac{817168432}{21}\\
    \displaystyle d(k) = -15\cdot \gamma(k)+70\cdot \beta(k)-120 \cdot \alpha(k)+16320 \\
   =  -15\cdot[ \frac{31}{168}\cdot 2^{3k}-\frac{1519}{6}\cdot2^{2k}+\frac{324976}{3}\cdot 2^k-\frac{300301312}{21}]\\
+70\cdot[ \frac{465}{32}\cdot2^{2k} -9396\cdot 2^k +1455744] -120\cdot  [\frac{2667}{16}\cdot 2^k -39228 ] +16320\\
= -\frac{465}{168}\cdot2^{3k}+\frac{231105}{48}\cdot2^{2k}-\frac{4605205}{2}\cdot2^k+\frac{2247886880}{7}\\
 \displaystyle e(k) = 70\cdot \gamma(k)-120\cdot \beta(k) +64\cdot \alpha(k)  \\
 = 70\cdot [\frac{31}{168}\cdot 2^{3k}-\frac{1519}{6}\cdot2^{2k}+\frac{324976}{3}\cdot 2^k-\frac{300301312}{21}]\\
 -120\cdot [\frac{465}{32}\cdot2^{2k} -9396\cdot 2^k +1455744] +64\cdot [ \frac{2667}{16}\cdot 2^k -39228] \\
  = \frac{155}{12}\cdot2^{3k}-\frac{233585}{12}\cdot2^{2k}+\frac{26162884}{3}\cdot2^k-\frac{3534612736}{3} \\
   \displaystyle f(k) = -120 \cdot \gamma(k) +64\cdot \beta(k) \\
  =  -120\cdot[\frac{31}{168}\cdot 2^{3k}-\frac{1519}{6}\cdot2^{2k}+\frac{324976}{3}\cdot 2^k-\frac{300301312}{21}] \\
  +64\cdot [\frac{465}{32}\cdot2^{2k} -9396\cdot 2^k +1455744] \\
   = -\frac{155}{7}\cdot2^{3k}+31310\cdot2^{2k}-13600384\cdot2^k+\frac{12012052480}{7} +93167616\\
   =  -\frac{155}{7}\cdot2^{3k}+31310\cdot2^{2k}-13600384\cdot2^k+\frac{1466315\cdot2^{13}}{7} +11373\cdot2^{13}\\
   \displaystyle g(k) = 64\cdot  \gamma(k) \\
  =64\cdot[\frac{31}{168}\cdot 2^{3k}-\frac{1519}{6}\cdot2^{2k}+\frac{324976}{3}\cdot 2^k-\frac{300301312}{21}]\\
  =\frac{248}{21}\cdot2^{3k}-\frac{48608}{3}\cdot2^{2k}+\frac{20798464}{3}\cdot2^k-\frac{293263\cdot2^{16}}{21}
  \end{cases}
    \end{equation}
\end{proof}

  \subsection{Computation of $ \Gamma_{i}^{\left[2\atop{\vdots \atop 2}\right]\times 10} for \; 0\leqslant i\leqslant \inf(2n,10)$.}
 We shall need the following Lemma : \\
  \begin{lem}
\label{lem 3.3}
 \begin{equation}
  \label{eq 3.14}
 \begin{cases} 
 \displaystyle  \Gamma_{0}^{\left[2\atop{\vdots \atop 2}\right]\times k}  = 1 \quad \text{if} \quad  k\geqslant 1 \\
\displaystyle  \Gamma_{1}^{\left[2\atop{\vdots \atop 2}\right]\times k}  = (2^{n}-1)\cdot 3 \quad \text{if} \quad  k\geqslant 2 \\
\displaystyle \Gamma_{2}^{\left[2\atop{\vdots \atop 2}\right]\times k} = 7\cdot2^{2n}+(2^{k+1}-25) \cdot 2^{n}-2^{k+1}+18 \quad \text{for} \quad k\geqslant 3\\
\displaystyle  \Gamma_{3}^{\left[2\atop{\vdots \atop 2}\right]\times k} = 15\cdot2^{3n} + (7\cdot2^k-133)\cdot2^{2n}+ (294-21\cdot 2^k) \cdot 2^{n}   -176+14\cdot2^k \quad \text{for} \quad k\geqslant 4\\
\displaystyle  \Gamma_{4}^{\left[2\atop{\vdots \atop 2}\right]\times k} = 31\cdot2^{4n} + \frac{35\cdot2^{k}-1210}{2}\cdot2^{3n}
+ \frac{2^{2k+2}-783\cdot2^{k}+19028}{6}\cdot 2^{2n}\\
\displaystyle +(-2^{2k+1}+269\cdot2^{k}-5744)\cdot2^n 
 +\frac{2^{2k+2}-117\cdot2^{k+2}+9440}{3}  \quad \text{for} \quad k\geqslant 5\\
  \displaystyle  \Gamma_{5}^{\left[2\atop{\vdots \atop 2}\right]\times k} = 63\cdot2^{5n} + (\frac{155}{4}\cdot2^{k}-2573)\cdot2^{4n}
+ (\frac{5}{2}\cdot2^{2k}-\frac{2565}{4}\cdot2^{k}+29150)\cdot2^{3n}\\
\displaystyle +\frac{1}{2}\cdot(-35\cdot2^{2k}+6265\cdot2^{k}-247520)\cdot2^{2n} 
\displaystyle +(35\cdot2^{2k}-5490\cdot2^{k}+203872)\cdot2^{n}\\
-20\cdot2^{2k}+2960\cdot2^{k}-106752  \quad \text{for} \quad k\geqslant 6\\
 \displaystyle  \Gamma_{6}^{\left[2\atop{\vdots \atop 2}\right]\times k} =   127 \cdot 2^{6n} +[ 651\cdot 2^{k-3}-10605 ] \cdot 2^{5n}+ [\frac{155}{3}\cdot 2^{2k-3} -22661\cdot 2^{k-3}  +\frac{748154}{3}] \cdot 2^{4n} \\
  +  \frac{1}{168}\cdot [2^{3k+3}-16723\cdot2^{2k}+5026378 \cdot 2^{k}-382091648 ]   \cdot 2^{3n} \\
   +[ - \frac{1}{3} \cdot 2^{3k} + \frac{5649}{12}\cdot 2^{2k} -\frac{368711}{3}\cdot 2^{k} +8753120 ] \cdot 2^{2n} \\
   +[\frac{2}{3} \cdot 2^{3k} - \frac{2437}{3}\cdot 2^{2k} + \frac{597736}{3}\cdot 2^{k} -\frac{41276672}{3}] \cdot 2^{n}\\
    -8\cdot [ \frac{1}{21}\cdot 2^{3k}-\frac{163}{3} \cdot 2^{2k}+\frac{38816}{3} \cdot 2^k-\frac{18483200}{21} ]    \quad \text{for} \quad k\geqslant 7\\
      \displaystyle  \Gamma_{7}^{\left[2\atop{\vdots \atop 2}\right]\times k} =   255\cdot 2^{7n}  +a(k) \cdot 2^{6n}  +b(k) \cdot 2^{5n}+ c(k) \cdot 2^{4n} + d(k)\cdot 2^{3n}+e(k) \cdot 2^{2n}  +f(k)\cdot 2^{n} +g(k) \\ 
  =  255\cdot 2^{7n}  + [\frac{2667}{16}\cdot 2^k -43053]  \cdot 2^{6n}  +[ \frac{465}{32}\cdot2^{2k}-\frac{190341}{16}\cdot2^k+2062014] \cdot 2^{5n}\\
  +  [ \frac{31}{168}\cdot 2^{3k}-\frac{45229}{96}\cdot2^{2k}+\frac{6262403}{24}\cdot2^k-\frac{817168432}{21}]\cdot 2^{4n} \\
  +[ -\frac{465}{168}\cdot2^{3k}+\frac{231105}{48}\cdot2^{2k}-\frac{4605205}{2}\cdot2^k+\frac{2247886880}{7}] \cdot 2^{3n}\\
  + [\frac{155}{12}\cdot2^{3k}-\frac{233585}{12}\cdot2^{2k}+\frac{26162884}{3}\cdot2^k-\frac{3534612736}{3} ] \cdot 2^{2n} \\
   + [-\frac{155}{7}\cdot2^{3k}+31310\cdot2^{2k}-13600384\cdot2^k+\frac{1466315\cdot2^{13}}{7} +11373\cdot2^{13}]\cdot 2^{n} \\
   + \frac{248}{21}\cdot2^{3k}-\frac{48608}{3}\cdot2^{2k}+\frac{20798464}{3}\cdot2^k-\frac{293263\cdot2^{16}}{21} \quad \text{for} \quad k\geqslant 8
   \end{cases}
 \end{equation}
 \begin{equation}
 \label{eq 3.15}
  \begin{cases}  
\displaystyle  \sum_{i = 0}^{\inf (2n,k)} \Gamma_{i}^{\left[2\atop{\vdots \atop 2}\right]\times k}  = 2^{(k+1)n}, \\ 
  \displaystyle  \sum_{i = 0}^{\inf (2n,k)} \Gamma_{i}^{\left[2\atop{\vdots \atop 2}\right]\times k} 2^{-i}  = 2^{n+k(n-1)}+2^{(k-1)n}-2^{(k-1)n-k},\\
  \displaystyle \sum_{i = 0}^{\inf (2n,k)} \Gamma_{i}^{\left[2\atop{\vdots \atop 2}\right]\times k} 2^{-2i}  =
   2^{n+k(n-2)}+2^{-n+k(n-2)}\cdot[3\cdot2^k-3] +2^{-2n+k(n-2)}\cdot[6\cdot2^{k-1}-6] \\
   +2^{-3n+kn}-6\cdot2^{n(k-3)-k}+8\cdot2^{-3n+k(n-2)}.
\end{cases}
    \end{equation}
  \end{lem}  
 \begin{proof}
 Lemma 3.3 follows from Lemma 3.3 in Cherly [14] and \eqref{eq 3.9}
 \end{proof}
 We deduce from \eqref{eq 3.14} and \eqref{eq 3.15}  with k=10.
     \begin{equation}
  \label{eq 3.16}
  \Gamma_{i}^{\left[2\atop{\vdots \atop 2}\right]\times 10} = \begin{cases}
1 & \text{if  } i = 0,        \\
 (2^{n}-1)\cdot 3 & \text{if   } i=1,\\
 7\cdot2^{2n}+2023 \cdot 2^{n}-2030 & \text{if   }  i = 2,  \\
 15\cdot 2^{3n}+7035\cdot 2^{2n}
-21210\cdot 2^{n}+14160 & \text{if   }  i = 3,  \\
 31\cdot 2^{4n} +17315\cdot 2^{3n}+568590 \cdot 2^{2n}
-1827440\cdot2^{n}+1241504 & \text{if   }  i=4, \\ 
 63 \cdot 2^{5n}+37107 \cdot2^{4n} +1993950 \cdot2^{3n}\\
 -15266160 \cdot 2^{2n}+31282272 \cdot 2^n -18047232 & \text{if   }  i=5, \\
  127\cdot 2^{6n}+72723\cdot2^{5n}+ 4120830\cdot 2^{4n}  -24883824\cdot 2^{3n}\\
 +18602976 \cdot 2^{2n}+54302976\cdot2^{n} -52215808 & \text{if   }  i=6, \\
 255\cdot 2^{7n}  +127635 \cdot 2^{6n}  +5117310 \cdot 2^{5n} -67607280 \cdot 2^{4n}\\
    + 39863520\cdot 2^{3n}+1210256640 \cdot 2^{2n}  -3062415360\cdot 2^{n} +1874657280  & \text{if   }  i=7, \\
     511\cdot2^{8n} +a_{7}^{(8)}\cdot2^{7n}  +a_{6}^{(8)}\cdot2^{6n}  +a_{5}^{(8)}\cdot2^{5n}+ a_{4}^{(8)}\cdot 2^{4n} \\ +
 a_{3}^{(8)}\cdot 2^{3n}+a_{2}^{(8)} \cdot 2^{2n}+ a_{1}^{(8)}\cdot2^{n} +a_{0}^{(8)}& \text{if   }  i=8, \\
  1023\cdot2^{9n} +a_{8}^{(9)}\cdot2^{8n} +a_{7}^{(9)}\cdot2^{7n}  +a_{6}^{(9)}\cdot2^{6n}  +a_{5}^{(9)}\cdot2^{5n}+ a_{4}^{(9)}\cdot 2^{4n} \\ +
 a_{3}^{(9)}\cdot 2^{3n}+a_{2}^{(9)} \cdot 2^{2n}+ a_{1}^{(9)}\cdot2^{n} +a_{0}^{(9)}& \text{if   }  i=9. \\
  2^{11n}-1023\cdot2^{9n}+a_{8}^{(10)}\cdot2^{8n} +a_{7}^{(10)}\cdot2^{7n}  +a_{6}^{(10)}\cdot2^{6n}  +a_{5}^{(10)}\cdot2^{5n}+ a_{4}^{(10)}\cdot 2^{4n} \\ + a_{3}^{(10)}\cdot 2^{3n}+a_{2}^{(10)} \cdot 2^{2n}+ a_{1}^{(10)}\cdot2^{n} +a_{0}^{(10)}& \text{if   }  i=10. \\
\end{cases}    
  \end{equation}
where, \\
 \begin{equation}
  \label{eq 3.17}
 \begin{cases} 
 \displaystyle  \sum_{i = 0}^{10} \Gamma_{i}^{\left[2\atop{\vdots \atop 2}\right]\times 10}  = 2^{11n}, \\ 
  \displaystyle  \sum_{i = 0}^{10} \Gamma_{i}^{\left[2\atop{\vdots \atop 2}\right]\times 10} 2^{10-i}  = 2^{11n}+1023\cdot 2^{9n},\\
  \displaystyle \sum_{i = 0}^{10} \Gamma_{i}^{\left[2\atop{\vdots \atop 2}\right]\times 10} 2^{20-2i}  =
   2^{11n}+3069\cdot2^{9n} +3066\cdot2^{8n} +1042440\cdot2^{7n}. 
\end{cases}
    \end{equation}
 Combining \eqref{eq 3.16} and \eqref{eq 3.17} we compute $a_{i}^{(j)}$ in \eqref{eq 3.16}  for $8\leqslant j\leqslant 10 , \; 0 \leqslant i \leqslant  j-1 $\\
 and we obtain from  \eqref{eq 3.16}  \\
    \begin{equation}
  \label{eq 3.18}
  \Gamma_{i}^{\left[2\atop{\vdots \atop 2}\right]\times 10} = \begin{cases}
1 & \text{if  } i = 0,        \\
 (2^{n}-1)\cdot 3 & \text{if   } i=1,\\
 7\cdot2^{2n}+2023 \cdot 2^{n}-2030 & \text{if   }  i = 2,  \\
 15\cdot 2^{3n}+7035\cdot 2^{2n}
-21210\cdot 2^{n}+14160 & \text{if   }  i = 3,  \\
 31\cdot 2^{4n} +17315\cdot 2^{3n}+568590 \cdot 2^{2n}
-1827440\cdot2^{n}+1241504 & \text{if   }  i=4, \\ 
 63 \cdot 2^{5n}+37107 \cdot2^{4n} +1993950 \cdot2^{3n}\\
 -15266160 \cdot 2^{2n}+31282272 \cdot 2^n -18047232 & \text{if   }  i=5, \\
  127\cdot 2^{6n}+72723\cdot2^{5n}+ 4120830\cdot 2^{4n}  -24883824\cdot 2^{3n}\\
 +18602976 \cdot 2^{2n}+54302976\cdot2^{n} -52215808 & \text{if   }  i=6, \\
  255\cdot 2^{7n}  +127635 \cdot 2^{6n}  +5117310 \cdot 2^{5n}\\
  -67607280 \cdot 2^{4n} + 39863520\cdot 2^{3n}+1210256640 \cdot 2^{2n} \\
     -3062415360\cdot 2^{n} +1874657280  & \text{if   }  i=7, \\
     511\cdot2^{8n} +171955\cdot2^{7n} -897890\cdot2^{6n}  -38376240\cdot2^{5n}\\
     + 323250144\cdot 2^{4n} + 271514880\cdot 2^{3n}-436135\cdot2^{14} \cdot 2^{2n}\\
      + 242795\cdot2^{16}\cdot2^{n} -4445\cdot2^{21} & \text{if   }  i=8, \\
       1023\cdot2^{9n} -1533\cdot2^{8n} -517650\cdot2^{7n}  +1798320\cdot2^{6n} \\
  +78214752\cdot2^{5n}-559464192\cdot 2^{4n}  -783237120 \cdot 2^{3n}
    \\+200235\cdot2^{16} \cdot 2^{2n}-106680\cdot2^{18}\cdot2^{n}+ 480\cdot2^{25}& \text{if   }  i=9. \\
   2^{11n}-1023\cdot2^{9n}+1022\cdot2^{8n} +345440\cdot2^{7n}  -1028192\cdot2^{6n} \\
   -45028608\cdot2^{5n}+ 299663360\cdot 2^{4n}  + 494731264\cdot 2^{3n}\\
   -27432\cdot2^{18} \cdot 2^{2n}+ 57344\cdot2^{18}\cdot2^{n} -256\cdot2^{25}& \text{if   }  i=10. \\
\end{cases}    
  \end{equation}
  
    \begin{example}
  \textbf{Computation of  $ R_{q,n}^{(k)} $ in the case k=10, q=4\; (see \eqref{eq 3.2} and \eqref{eq 3.3}) }
     The number denoted by $ R_{4,n}^{(10)} $ of solutions \\
  
 $(Y_1,U_{1}^{(1)},U_{2}^{(1)}, \ldots,U_{n}^{(1)}, Y_2,U_{1}^{(2)},U_{2}^{(2)}, 
\ldots,U_{n}^{(2)}, Y_3,U_{1}^{(3)},U_{2}^{(3)}, \ldots,U_{n}^{(3)}, Y_4,U_{1}^{(4)},U_{2}^{(4)}, \ldots,U_{n}^{(4)}   ) \in (\mathbb{F}_{2}[T])^{4(n+1)}$ \vspace{0.5 cm}\\
 of the polynomial equations  \vspace{0.5 cm}
  \[\left\{\begin{array}{c}
 Y_{1}U_{1}^{(1)} + Y_{2}U_{1}^{(2)} + Y_{3}U_{1}^{(3)} +Y_{4}U_{1}^{(4)} = 0  \\
    Y_{1}U_{2}^{(1)} + Y_{2}U_{2}^{(2)}  + Y_{3}U_{2}^{(3)} + Y_{4}U_{2}^{(4)}= 0\\
    \vdots \\
   Y_{1}U_{n}^{(1)} + Y_{2}U_{n}^{(2)} + Y_{3}U_{n}^{(3)} + Y_{4}U_{n}^{(4)}= 0 
 \end{array}\right.\]
 
  satisfying the degree conditions \\
                   $$  degY_i \leq 9,
                   \quad degU_{j}^{(i)} \leq 1, \quad  for \quad 1\leq j\leq n,   \quad 1\leq i \leq 4 $$ \\
  is equal to

  \begin{align*}
 R_{q,n}^{(k)} =  2^{q(2n+k) - (k+1)n}\sum_{i = 0}^{\inf(2n,k)}
  \Gamma_{i}^{\left[2\atop{\vdots \atop 2}\right]\times k} 2^{-iq}
   = R_{4,n}^{(10)} = 
   2^{40-3n}\sum_{i = 0}^{10} \Gamma_{i}^{\left[2\atop{\vdots \atop 2}\right]\times 10} 2^{-i4}\\
    = 2^{40-3n}\cdot \bigg[ 2^{-40}\cdot 2^{11n} \\
     +(-1023\cdot 2^{-40}+1023\cdot 2^{-36})\cdot2^{9n}\\
     +(1022\cdot 2^{-40}-1533\cdot 2^{-36}+511\cdot 2^{-32})\cdot2^{8n}\\
  +(345440\cdot2^{-40}-517650\cdot 2^{-36}  + 171955\cdot 2^{-32}+255\cdot 2^{-28})\cdot2^{7n}\\
  +(-1028192\cdot2^{-40}+1798320\cdot2^{-36}  -897890\cdot2^{-32}\\+127635\cdot2^{-28}+127\cdot2^{-24})\cdot2^{6n}\\ 
  +(-45028608\cdot2^{-40}+78214752\cdot2^{-36}  -38376240\cdot2^{-32}\\
 +5117310\cdot2^{-28}+72723\cdot2^{-24}+63\cdot2^{-20})\cdot2^{5n}\\
  +(299663360\cdot2^{-40}-559464192\cdot2^{-36}  +323250144\cdot2^{-32}\\
  -67607280\cdot2^{-28}+4120830\cdot2^{-24}+37107\cdot2^{-20}+31\cdot2^{-16})\cdot2^{4n}\\
 +(494731264\cdot2^{-40}-783237120\cdot2^{-36}  +271514880\cdot2^{-32}\\
  39863520\cdot2^{-28}-24883824\cdot2^{-24}+1993950\cdot2^{-20}+17315\cdot2^{-16}+15\cdot2^{-12})\cdot2^{3n}
 \bigg] \\
  = 2^{8n}+15345\cdot 2^{6n}+107310\cdot 2^{5n}+37128000 \cdot 2^{4n}+329001120\cdot 2^{3n}\\+67088385\cdot2^{8}\cdot2^{2n}
  +26043255\cdot2^{12}\cdot2^{n}+2^{16}\cdot14881860.
  \end{align*}
 \textbf{ The case n=1:}\\
 
  $  R_{4,1}^{(10)}= 2^{8}+15345\cdot 2^{6}+107310\cdot 2^{5}+37128000 \cdot 2^{4}+329001120\cdot 2^{3}\\+67088385\cdot2^{8}\cdot2^{2}
  +26043255\cdot2^{12}\cdot2+2^{16}\cdot14881860=587\cdot2^{31}$\\
  Equally we obtain : \\ $ \displaystyle \quad R_{4,1}^{(10)}=   2^{37}\sum_{i = 0}^{2} \Gamma_{i}^{2 \times 10} 2^{-i4} = 
  2^{37}\cdot [1+3\cdot2^{-4}+2044\cdot2^{-8}] =587\cdot2^{31} $ see [3],[4]\\
 \textbf{ The case n=2:}\\
 
  $  R_{4,2}^{(10)}= 2^{16}+15345\cdot 2^{12}+107310\cdot 2^{10}+37128000 \cdot 2^{8}+329001120\cdot 2^{6}\\+67088385\cdot2^{8}\cdot 2^{4}
  +26043255\cdot2^{12}\cdot2^{2}+2^{16}\cdot 14881860=6361\cdot2^{28}$\\
  Equally we obtain : \\ $ \displaystyle \quad R_{4,2}^{(10)}=   2^{34}\sum_{i = 0}^{4} \Gamma_{i}^{\left[2\atop 2\right]\times 10} 2^{-i4}= 
  2^{34}\cdot [1+9\cdot2^{-4}+6174\cdot2^{-8}+42840\cdot2^{-12}+4145280\cdot2^{-16}] =6361\cdot2^{28} $ see [5]\\
  
 \textbf{ The case n=3:}\\
 
 $ R_{4,3}^{(10)}=
   2^{24}+15345\cdot 2^{18}+107310\cdot 2^{15}+37128000 \cdot 2^{12}+329001120\cdot 2^{9}\\+67088385\cdot2^{8}\cdot2^{6}
  +26043255\cdot2^{12}\cdot2^{3}+2^{16}\cdot14881860 = 1552553\cdot2^{21}$\\
   Equally we obtain : \\ $ \displaystyle \quad R_{4,3}^{(10)}=   2^{31}\sum_{i = 0}^{6} \Gamma_{i}^{\left[2\atop{ 2\atop 2}\right]\times 10} 2^{-i4}= 
  2^{31}\cdot [1+21\cdot2^{-4}+14602\cdot2^{-8}+302400\cdot2^{-12}+32004000\cdot2^{-16}+430133760\cdot2^{-20}+8127479808\cdot2^{-24}] =1552553\cdot2^{21} $ see [6]\\

  \end{example}

\end{document}